\documentclass{amsart}
\usepackage{amssymb,amsmath,amscd,amsthm}
\usepackage{amsfonts,epsfig,latexsym,graphicx,amssymb}
\usepackage{color}
\usepackage{enumerate}
\usepackage{todonotes}

\usepackage{hyperref}

\usepackage{tikz}
\date{\today}

\newcommand{\Z}{{\mathbb Z}}
\newcommand{\bbH}{{\mathbb H}}
\newcommand{\R}{{\mathbb R}}
\newcommand{\C}{{\mathbb C}}
\newcommand{\N}{{\mathbb N}}

\newcommand{\Lip}{\mathop{\mathrm{Lip}}}

\newcommand{\mK}{\mathcal{K}}
\newcommand{\mA}{\mathcal{A}}
\newcommand{\Cn}{M}

\newcommand{\OZ}{\Omega_{\Z}}

\newcommand{\RP}{{\mathbb{RP}}}
\newcommand{\CP}{{\mathbb{CP}}}

\newcommand{\SL}{{\mathrm{SL}}}
\newcommand{\SO}{{\mathrm{SO}}}

\graphicspath{{./pics-Hofstadter/}}

\newcounter{mit}

\newcounter{Bit}
\newenvironment{Bitemize}
	{\setcounter{Bit}{0}\begin{itemize}}
	{\end{itemize}}
\newcommand{\bitem}{\refstepcounter{Bit}\item[\theBit]}
\renewcommand{\theBit}{(B\arabic{Bit})}

\newtheorem{theorem}{Theorem}[section]
\newtheorem{lemma}[theorem]{Lemma}
\newtheorem{prop}[theorem]{Proposition}

\theoremstyle{definition}
\newtheorem{remark}[theorem]{Remark}

\newtheorem{defi}[theorem]{Definition}

\newcommand{\mP}{\mathcal{P}}
\newcommand{\mE}{\mathcal{E}}
\newcommand{\mmax}{\widehat{m}}

\def\diam{{\rm diam}}
\def\dist{{\rm dist}}
\def\diam{{\rm diam}}
\def\supp{\mathop{\rm supp}}

\newcommand{\bi}{{\bf i}}

\def\N{{\mathbb N}}

\newcommand{\E}{{\mathbb E}\,}
\newcommand{\Prob}{{\mathbb P}\,}

\def\be{\begin{equation}}
\def\ee{\end{equation}}

\newcommand{\eps}{{\varepsilon}}

\newcommand{\mgr}{\lambda}
\newcommand{\msp}{\mu}

\title[Random Schr\"odinger Operators Defined by Block Factors]{Localization for Random Schr\"odinger Operators Defined by Block Factors}

\author[D.\ Damanik]{David Damanik}

\address{Department of Mathematics, Rice University, Houston, TX~77005, USA}

\email{damanik@rice.edu}

\thanks{D.\ D.\ was supported in part by NSF grants DMS--2054752 and DMS--2349919.}

\author[A.\ Gorodetski]{Anton Gorodetski}

\address{Department of Mathematics, University of California, Irvine, CA~92697, USA}

\email{asgor@uci.edu}

\thanks{A.\ G.\ was supported in part by NSF grant DMS--2247966.}

\author[V. Kleptsyn]{Victor Kleptsyn}

\address{CNRS, Institute of Mathematical Research of Rennes, IRMAR, UMR 6625 du CNRS}

\email{victor.kleptsyn@univ-rennes1.fr}

\thanks{V.K. was supported in part by ANR Gromeov (ANR-19-CE40-0007) and by Centre Henri Lebesgue (ANR-11-LABX-0020-01)}

\begin{document}

\begin{abstract}
We consider discrete one-dimensional Schr\"odinger operators with random potentials obtained via a block code applied to an i.i.d.\ sequence of random variables. It is shown that, almost surely, these operators exhibit spectral and dynamical localization, the latter away from a finite set of exceptional energies. We make no assumptions {beyond non-triviality}, neither on the regularity of the underlying random variables, nor on the linearity, the monotonicity, or even the continuity of the block code. Central to our proof is a reduction to the non-stationary Anderson model via Fubini.
\end{abstract}

\maketitle

\section{Introduction}
Anderson localization is the key phenomenon in the spectral theory of random Schr\"odinger operators. The one-dimensional Anderson model is given by the discrete Schr\"odinger operator
 \begin{equation}\label{e.1DAnderson}
[H\psi] (n) = \psi(n+1) + \psi(n-1) + V(n) \psi(n),
\end{equation}
where $\{V(n)\}_{n\in \Z}$ is a sequence of i.i.d.\ random variables. The localization statement for this operator typically takes two forms. Spectral localization means that almost surely, the operator $H$ has pure point spectrum with exponentially decaying eigenfunctions. Dynamical localization means that the unitary group associated with $H$ has exponential off-diagonal decay relative to the standard orthonormal basis $\{ \delta_n \}_{n \in \Z}$ of $\ell^2(\Z)$ uniformly in time, either in an almost sure sense or in expectation. Spectral localization and dynamical localization are not equivalent for general operators, but they both hold for the Anderson model, see \cite[Chapter~5]{DF24} for a review of these classical results.

The first results on Anderson localization were established for potentials given by i.i.d.\ random variables given by a regular (e.g. absolutely continuous) distribution, e.g. see \cite{KuS, Mol, GMP}. It is natural to ask whether the assumptions on regularity, identical distribution, and even independence of distributions that define the random potential can be relaxed.

Proving the localization statements for the Anderson model is most difficult in the case of a singular single-site distribution, for example the Bernoulli case, where the {support of the} distribution that defines that potential has cardinality~$2$. Indeed,
%In fact it is true that the Bernoulli case is the most difficult case to handle, as
any existing localization proof that covers the Bernoulli case does also cover the general case. We refer the reader to \cite{CKM} for the first proof of spectral localization in the Bernoulli case, to \cite{SVW} for the second proof, and to \cite{BDFGVWZ, GZ20, GK1, JZh} for recent treatments of this model, which establish both spectral and dynamical localization and are simpler and more conceptual (in that they rely on one-dimensional tools, rather than verifying the input necessary to run a multi-scale analysis).

It turns out that the assumption that the random variables $V(n)$ are identically distributed can also be removed. For a sufficiently regular distribution this can be deduced, for example, from the original Kunz-Soulliard approach \cite{KuS}. For the general 1D Anderson model, including the non-stationary Anderson-Bernoulli case, both spectral and dynamical localization were shown in \cite{GK3}. An alternative path to the proof of spectral localization in the non-stationary case can be found in [SVW]; it passes through estimates for the Green’s function, followed by the usual multi-scale arguments (see \cite[Proposition 3.6]{SVW} and
the remark after it).

On the other hand, one can try to relax the assumption on the  independence of $\{V(n)\}$. Some results in that direction were obtained already in the 1990s, see \cite{DK91}, \cite{AM}. In particular, the so-called alloy-type Anderson model, where the potential is represented as a linear combination of independent random variables indexed by the sites of the lattice with fast decaying coefficients, attracted considerable attention \cite{BK}, \cite{Kir96}, \cite{TV}, \cite{St}, \cite{Kr}, \cite{ETV}, \cite{EKTV}, \cite{ESS}. In \cite{Kl13}, the so-called ``crooked alloy-type random potentials'' are considered; a related ``random displacements model'' was studied in \cite{Ch}.

Here we consider the class of random potentials defined by a block factor without any assumptions on its linearity, monotonicity, or even continuity:

\begin{theorem}\label{t.main}
Let $\{\xi_n\}_{n\in\Z}$ be an i.i.d.\ sequence of random variables defined by a Borel probability distribution $\nu$. Assume that %$\# \, \mathrm{supp} \, \nu \ge 2$, $k \in \N$,
the function $g : \R^k \to \R$, $k \in \N$, is bounded, Borel measurable, and essentially non-constant, that is, not {equal to} a constant $\nu^k$-almost everywhere. % and such that when freezing all but the first or the last variable, we obtain a non-constant analytic function of that single variable.
Then, $\nu^{\Z}$-almost surely, the random operator $H: \ell^2(\Z) \to \ell^2(\Z)$ defined by
 \begin{equation}\label{e.model}
 [H \psi] (n) = \psi(n+1) + \psi(n-1) + v_n \psi(n), \ v_n=g(\xi_n, \xi_{n+1}, \ldots, \xi_{n+k-1}),
 \end{equation}
has pure point spectrum {with exponentially decaying eigenfunctions}.
\end{theorem}

In Theorem \ref{t.main} one cannot in general add almost sure exponential decay of eigenfunctions with a uniform rate, which does hold in the {stationary and even the non-stationary} Anderson model. This is due to the fact that in some cases there could be exceptional energies where the Lyapunov exponent vanishes; we give an explicit example in Proposition \ref{p.example} below. {Such energies will belong to the almost sure spectrum, and hence (by upper semi-continuity of the Lyapunov exponent), the exponential decay rate must converge to zero for eigenfunctions with eigenvalues approaching this exceptional energy.} This phenomenon is {akin to the one observed earlier} in the random dimer model \cite{DWP} (see also \cite[Section~5.11]{DF24}) and in the case of potentials defined by a hyperbolic base dynamics \cite{ADZ1}. {Nevertheless, since individual energies are not themselves eigenvalues with probability one, we can still conclude the existence of a basis of exponentially decaying eigenvectors for almost all realizations of the potential.}

{On the other hand, the presence of exceptional energies does mean that we can only get dynamical localization away from them}:

\begin{theorem}\label{t.dynloc}
Under the assumptions of Theorem \ref{t.main}, there exists a finite set $\mathcal{E} \subset \R$ such that for any compact $J \subset \R \setminus \mathcal{E}$, there is $\beta > 0$ such that for any $\varepsilon > 0$, the following holds: $\nu^{\Z}$-almost surely, there is a constant $C> 0$ such that
\begin{equation}\label{eq:edl}
	\sup\limits_{t \in \R}
	|\langle\delta_n, e^{-itH} \chi_{J}(H) \delta_m\rangle|
	\leq C e^{\epsilon |m|} e^{-\beta|n-m|}
\end{equation}
for all $m,n \in \Z$.
\end{theorem}

%\begin{remark}\label{r.N}
%Analogs of Theorems \ref{t.main} and \ref{t.dynloc} for the operators on $\ell^2(\N)$ also hold, with straightforward modifications of the proofs.
%\end{remark}

The strategy of the proof of Theorems \ref{t.main} and \ref{t.dynloc}, which, we believe, can be adapted to many other models (we mention some in Remark \ref{r.r} below) consists of two major steps. On a hand waving level, the first one is a reduction of a model with random potentials exhibiting finite range dependencies to the case of independent but not identically distributed matrix valued random variables. Namely, let us split the lattice $\Z$ into finite blocks, and freeze the values of $\{\xi_n\}$ in, say, every other block. If the size of the frozen blocks is sufficiently large, the potential values $\{V(n)\}$ also split into ``random'' and ``frozen'' blocks, and  the products of random transfer matrices over different random blocks are independent, but, certainly, not identically distributed, since their laws depend on the values of the frozen variables. If we are able to prove localization for the random model with some frozen values of $\{\xi_n\}$, then by Fubini this implies localization for the initial model almost surely. In order to prove localization in the non-stationary case, one can use the approach derived in \cite{GK3}. One of the conditions that has to be checked is the so called {\it measure condition} (see condition (B1) in Section \ref{ss.introparamfurst}). The second main step is to show that the measure condition must be satisfied for all but a finite number of energies. Once again, on a hand waving level, the idea is to show that the set of energies where this condition is not satisfied is actually algebraic, hence must be the whole line or finite. Since one can show explicitly that for large values of the energy the condition must be satisfied, the set of exceptional energies has to be finite. Another condition that has to be checked is the compactness of the ``pool'' of distributions required in \cite{GK3}. We address that in Lemma~\ref{l:compact}.

\begin{remark}\label{r.r}

\

\begin{enumerate}[a)]
\item Analogs of Theorems \ref{t.main} and \ref{t.dynloc} for half-line operators on $\ell^2(\N)$ also hold, with straightforward modifications of the proofs.

\item The proof of Theorems \ref{t.main} and \ref{t.dynloc} does not use the assumption that the background random variables $\{\xi_n\}$ are real valued in any way. One could replace $\nu$ by a probability distribution on any measure space $X$, take $\{\xi_n\}$ to be a sequence of i.i.d.\ random variables with values in $X$ chosen with respect to that distribution, and take the block function $g:X^k\to \R$ to be any measurable bounded essentially non-constant function.

\item There are other scenarios where a similar approach could be helpful. For example,
one can replace the i.i.d.\ sequence $\{\xi_n\}$ in Theorem \ref{t.main} by a sequence of mutually independent but not identically distributed random variables. In particular, if the collection of distributions we can choose from is finite, an almost verbatim repetition of the proofs of Theorems \ref{t.main} and \ref{t.dynloc} can be applied. Or one can consider a sequence of block maps $g_n:\R^k\to \R$, and define the potential by $v_n=g_n(\xi_n, \xi_{n+1}, \ldots, \xi_{n+k-1})$. It is reasonable to expect that under some suitable conditions that guarantee that the randomness of the potential formed in this way  does not degenerate, analogs of Theorems \ref{t.main} and \ref{t.dynloc} hold.

\item One could consider a block map $g$ of the form $g:\R^k\to \R^m$, and form a potential by concatenation of the blocks produced by the block map. This model would cover, in particular, the random dimer model.

\item The potential in (\ref{e.model}) belongs to the class of ergodic (or dynamically defined) potentials in the sense of \cite{DF22, DF24}.  Indeed, it can be modeled by considering the compact product $\Omega = (\mathrm{supp} \, \nu)^\Z$, the shift transformation $T : \Omega \to \Omega$, $[T \omega](n) = \omega(n+1)$, and the $T$-ergodic measure $\nu^\Z$ on $\Omega$, and the potential can then be written in the form
\begin{equation*}\label{eq:dyndefoper}
v_{n, \omega} = f(T^n \omega),
\end{equation*}
where in the case of the classical Anderson model the sampling function $f : \Omega \to \R$ is given by the evaluation at the origin,
\begin{equation*}\label{eq:AMsamplingfunction}
f(\omega) = \omega(0),
\end{equation*}
while for the potential (\ref{e.model}) given by a block code we have
\begin{equation}\label{eq:AMsamplingfunction2}
f(\omega) = g(\omega(0), \ldots, \omega(k-1)).
\end{equation}

\item In the case where $\mathrm{supp} \, \nu$ is finite, the paper \cite{ADZ1} has developed a method, which in fact applies in greater generality beyond the case of the full shift, to show the positivity of the Lyapunov exponent away from a finite set of energies as soon as the sampling function is locally constant and non-constant on $(\mathrm{supp} \, \nu)^\Z$  (i.e., it is of the form \eqref{eq:AMsamplingfunction2} with a $g$ that is not constant on $(\mathrm{supp} \, \nu)^k$). Localization statements may then be derived from this input, as shown in a follow-up work by the same authors \cite{ADZ2}.

\item The statement of Theorem \ref{t.main} is quite general, and does not require any regularity of the distribution $\nu$, or linearity or monotonicity of the function $g$. At the same time, we would like to point out that there are two assumptions that one can reasonably expect to be unnecessary. First, it is an assumption that $g$ is bounded. One should expect that it can potentially be generalized to the case when $v_n=g(\xi_n, \xi_{n+1}, \ldots, \xi_{n+k-1})$ has some finite  moment, i.e.
     $$
     \int |g(\xi_1, \xi_{2}, \ldots, \xi_{k})|^\gamma d\nu^{k}<\infty \text{\ for some }\ \gamma>0.
     $$
     Second, it is an assumption that the potential in our setting has only finite range dependencies. Notice that some results in the case of alloy type Anderson model do allow infinite range dependencies, which in our case would mean that $g$ is a function which ``dependence'' on $\xi_i$ tends to zero sufficiently fast as $|i|\to \infty$ in some sense. In fact, in the case $\# \, \mathrm{supp} \, \nu < \infty$, the work \cite{ADZ1, ADZ2} can deal with H\"older continuous sampling functions $f : ( \mathrm{supp} \, \nu )^\Z \to \R$.

\end{enumerate}
\end{remark}

\section{Parametric Non-Stationary Furstenberg Theorem}\label{ss.introparamfurst}

A central component of the proof of Theorem~\ref{t.main} is aimed at understanding the properties of the products of transfer matrices for the operators in question, and the way those products behave for different values of the energy. We will rely on the recent work \cite{GK3} by two of the present authors. Let us recall a key theorem from \cite{GK3} which we will invoke in the next section when proving Theorem~\ref{t.main}.

Let us consider random products of independent but not identically distributed matrices from $\SL(2, \mathbb{R})$ that depend on a parameter. That is, we are working with maps $A(\cdot)$ from some compact interval $J=[b_{-}, b_{+}]\subset \R$ to $\SL(2,\R)$. We assume that all these maps are $C^1$. A random matrix depending on a parameter is therefore given by a measure on the space $\mA:=C^1(J,\SL(2,\R))$.  %, where $J\subset \R$ is a compact interval of parameters.
For any such measure $\mgr$ on $\mA$ and any individual parameter value $a\in J$, we can consider the distribution of $A(a)$, which is a measure on~$\SL(2,\R$); we denote this measure by~$\mgr^a$.

For any $A\in \SL(2, \R)$ we denote by $f_A:\mathbb{RP}^1\to \mathbb{RP}^1$ the corresponding projective map. The measure $\mgr^a$ therefore defines a distribution on the space of projective maps, which, slightly abusing notation, we will denote by the same symbol.

A (non-stationary) product of random matrices, depending on a parameter, is given by a sequence of measures $\mgr_n$ on $\mA$. We assume that all these measures belong to some compact set $\mK$ of measures on $C^1(J,\SL(2,\R))$, i.e. $\mgr_n\in \mK$ for all~$n\in \mathbb{Z}$.

We impose the following assumptions:
\begin{Bitemize}
\bitem\label{B:Furstenberg} \textbf{Measures condition}: for any measure $\mgr\in\mK$ and any $a\in J$, there are no Borel probability measures $\msp_1$, $\msp_2$ on $\mathbb{RP}^{1}$ such that $(f_A)_*\msp_1=\msp_2$ for $\mgr^a$-almost every matrix $A\in \SL(2, \mathbb{R})$.
 %law $\mgr^a$ of~$A(a)$, where $A$ is distributed w.r.t.~$\mgr$.
\bitem\label{B:C1} \textbf{$C^1$-boundedness}: there exists a constant $\Cn$ such that any map $A(\cdot) \in C^1(J,\SL(2,\R))$
from the support of any $\mgr\in\mK$ has $C^1$-norm at most $\Cn$.
\bitem\label{B:Monotonicity} \textbf{Monotonicity}: there exists $\delta>0$ such that for any $\mgr\in\mK$, any map $A(\cdot)$ from the support of $\mgr$, and any $a_0\in J$, one has
$$
%\forall a_0\in J, \quad
\forall v\in \R^2\setminus \{0\} \quad \left.\frac{d\arg (A(a)(v))}{da}\right|_{a_0}  >\delta.
$$
\end{Bitemize}

Consider the product space $\OZ:=\mA^{\Z}$ and the product measure $\Prob:=\prod_{n\in \Z} \mgr_n$. For $n \in \N$, $a \in J$, and $\omega=(\dots, A_{-1},A_0, A_1,\dots)\in\OZ$, we denote
$$
\mathbf{T}_{n,\omega, a}:=A_n(a)\dots A_1(a),
$$
$$
\mathbf{T}_{-n, \omega, a}:= (A_{-n+1}(a))^{-1} \dots (A_{-1}(a))^{-1} (A_{0}(a))^{-1},
$$
and
$$
L_{n}(a):=\E \log \|\mathbf{T}_{n,\omega, a} \|,\ L_{-n}(a):=\E \log \|\mathbf{T}_{-n,\omega, a} \|.
$$
Then we have the following statement, which combines results from \cite{GK1} and \cite{GK3}:

\begin{theorem}\cite[Theorem 1.1]{GK1}\cite[Theorem 1.14]{GK3}\label{t.vector}
Under the assumptions~\ref{B:Furstenberg}--\ref{B:Monotonicity} we have:
\begin{itemize}

\item The sequence $L_n(a) =\E \log \|\mathbf{T}_{n,  \omega, a}\|$ grows at least linearly, that is, there exists
$h > 0$ such that for any $n\in \mathbb{N}$, $a\in J$, and any $\mu_1, \mu_2, \ldots, \mu_n\in \mK$, we have
$L_n(a)\ge hn$.

\item For almost all $\omega\in\OZ$, the following holds for all $a\in J$: If
%\marginpar{For the purpose of the present paper, can't the second item be removed? Well, if we want to be able to claim that Theorem \ref{t.main} holds on the half line as well (see Remark \ref{r.r}, part a)), we probably want to keep it...}
$$%\begin{equation}\label{eq:lim-less}
\limsup_{n\to+\infty} \frac{1}{n} \left( \log |\mathbf{T}_{n,  \omega, a} \left( \begin{smallmatrix} 1 \\ 0 \end{smallmatrix} \right)| - L_n(a) \right)<0,
$$%\end{equation}
then in fact $|\mathbf{T}_{n,a, \omega} \left(  \begin{smallmatrix} 1 \\ 0 \end{smallmatrix} \right)|$ tends to zero exponentially as $n\to \infty$. Namely, %, as fast as it can:
$$
 \log |\mathbf{T}_{n, \omega, a} \left( \begin{smallmatrix} 1 \\ 0 \end{smallmatrix} \right)| = -L_n(a) + o(n).
$$

\item[$\bullet$] For almost all $\omega\in\OZ$, the following holds for all $a\in J$: If for some $\bar{u}\in \R^2\setminus \{0\}$, we have
$$%\begin{equation}\label{eq:lim-both-less}
\limsup_{n\to+\infty} \frac{1}{n} \left( \log |\mathbf{T}_{n, \omega,a} \bar u| - L_n(a) \right)<0
$$
and
$$
\limsup_{n\to+\infty} \frac{1}{n} \left( \log |\mathbf{T}_{-n, \omega, a} \bar u| - L_{-n}(a) \right)<0,
$$%\end{equation}
then both sequences $|\mathbf{T}_{n, \omega, a} \bar u|, |\mathbf{T}_{-n, \omega, a} \bar u|$ in fact tend to zero exponentially. More specifically,
$$
\log |\mathbf{T}_{n, \omega, a} \bar u| = - L_n(a) + o(n)
$$
and
$$
\log |\mathbf{T}_{-n, \omega, a} \bar u| = - L_{-n}(a) + o(n).
$$
\end{itemize}
\end{theorem}

\begin{remark}
The Monotonicity Condition (B3) can certainly be replaced by an assumption that $\arg (A(a)(v))$ decreases (instead of assumption that it increases). Therefore, for the random products of inverses of the matrices $\{A_{i}(a)\}$ the expectation  $L_{-n}(a) =\E \log \|\mathbf{T}_{-n,  \omega, a}\|$ grows at least linearly as well.
\end{remark}

In what follows, the role of the parameter $a\in J$ will be played by the energy $E$. Let us recall that for a given potential $\{v_n\}$ one defines transfer matrices by
$$
\Pi_{n,E}=\begin{pmatrix}
            E-v_n & -1 \\
            1 & 0 \\
          \end{pmatrix},
$$
and, similarly to the notation above, we denote their products by
$$
T_{N, E}=\left\{
           \begin{array}{lll}
             \Pi_{N,E}\ldots \Pi_{2,E}\Pi_{1,E},  & \hbox{if $N\ge 1$;} \\
             \text{\rm Id} & \hbox{if $N=0$;}\\
             \Pi_{-N,E}^{-1}\ldots \Pi_{-1,E}^{-1}\Pi_{0,E}^{-1}, & \hbox{if $N<0$,}
           \end{array}
         \right.
$$
and
$ T_{[N_1, N_2], E}=T_{N_2, E}T_{N_1-1, E}^{-1}$.

\section{The Reduction to the Non-Stationary Case}

Let $\{\xi_n\}_{n\in \Z}$ be an i.i.d.\ sequence of random variables, and let the potential $v_j$ depend on $\xi_{j},\dots,\xi_{j+k-1}$:
$$v_j=g(\xi_{j},\dots,\xi_{j+k-1}).$$
Fix the values of $\xi_{-k-1},\dots,\xi_{-1},\xi_1,\dots,\xi_{2k+1}$. Then we get a random vector of potential values $v_{-k+1},\dots,v_{0}$, depending on $\xi_0$ only, which is independent of the (random) values of the potential with indices outside of this range. So altering ``fixed'' and ``random'' strings, we get a non-stationary product of \emph{independent} (transfer) matrices. This puts us in the setting of Theorem \ref{t.vector}, which deals with products of independent (but not necessarily identically distributed) random matrices. In order to be able to apply it, we need to verify that all the random matrices are defined by distributions from some compact set of measures such that each distribution from that set satisfies the conditions (B1)--(B3). In order to do that, it is actually convenient to form even larger blocks of transfer matrices.

The following is the key statement:
\begin{prop}\label{p:main}
Assume that the law of the random vector $\vec{v}=(v_{1},\dots,v_{d})$  has at least five points in its support, {for which the following property holds: for some intermediate index $i_0$, any two of these vectors differ in at least one position $i<i_0$ and in at least one position $i>i_0+1$.}
% satisfies the following (quite generic) property:
%, such that all the first coordinates are different, and all the last coordinates are different.
Then there exists a finite set $X\subset \R$ of energies, such that for any $E\notin X$, the measure condition is satisfied for the corresponding products of transfer matrices for that value of~$E$, i.e. there are no two probability measures $\nu_1, \nu_2$ on $\mathbb{RP}^1$ such that projectivizations of all the corresponding products of transfer matrices would send $\nu_1$ to $\nu_2$.
\end{prop}

\begin{remark}
In fact, in Proposition \ref{p:main} one can even provide an explicit upper bound for the cardinality of the set $X$ (which will depend only on the size $d$ of the random vector $\vec{v}$, and will grow at most linearly with $d$).
\end{remark}

We will give the proof of Proposition \ref{p:main} in Section \ref{s.proofofprop}.

\medskip

{
Note that Proposition \ref{p:main} is applicable under the assumptions of Theorem~\ref{t.main}.
\begin{lemma}\label{l:p-applicable}
Assume that the assumptions of Theorem~\ref{t.main} are satisfied. Then, conditionally to any values of $\xi_{k+1}, \dots, \xi_{2k}$, $\xi_{9k+1}, \dots, \xi_{10k}$, the law of
\[
\vec{v}=(v_{k+1},\dots, v_{9k})\in \R^{8k}
\]
satisfies the assumptions of Proposition~\ref{p:main}.
\end{lemma}
}
\begin{proof}
Note that $v_{2k+1}$, $v_{3k+1}$, $v_{4k+1}$, $v_{6k+1}$, $v_{7k+1}$, $v_{8k+1}$ are i.i.d.\ (see Fig.\ref{f:block} for the illustration of dependencies), and their distribution is given by the measure $m:=g_* \nu^k$, which is not concentrated on a single point by the assumption of Theorem~\ref{t.main}. Hence, the distribution $m$ of these coordinates has at least two different points $u_1,u_2$ in its support. Thus, the support of $m^3$ contains at least 8 points (belonging to $\{u_1,u_2\}^3$), and hence the support of $m^6$ contains at least~8~points of the form
\[
\{(u,u) \mid u\in \supp m^3 \}.
\]
Every two of these 8 points differ in at least one of the first three coordinates, and in at least one of the last three. Now, for each of these 8 points, there is a point in the support of the law of $\vec{v}$ that projects to it when forgetting all the other coordinates. Pick these 8 points, select any 5 among them, and take $i_0:=5k$; we see that the assumptions of Proposition~\ref{p:main} are satisfied.
\end{proof}

%\newpage

We are now going to ensure the applicability of the Nonstationary Furstenberg Theorem via the scheme of freezing some of the $\xi_i$'s, as discussed in the introduction. To do this, decompose a group of $11k$ consecutive $\xi_i$'s, from $\xi_1$ to $\xi_{11k}$, into
\begin{equation}\label{eq:Xi-1}
\Xi_1:= (\Xi_1^-, \Xi_1^+), \quad  \Xi_1^-= (\xi_1,\dots,\xi_k), \quad \Xi_1^+=(\xi_{10k+1},\dots,\xi_{11k}),
\end{equation}
\begin{equation}\label{eq:Xi-2}
\Xi_2:= (\Xi_2^-, \Xi_2^+), \quad  \Xi_2^-=(\xi_{k+1},\dots,\xi_{2k}), \quad \Xi_2^+=(\xi_{9k+1},\dots,\xi_{10k}),
\end{equation}
and
\[
\Xi_3:=(\xi_{2k+1},\dots,\xi_{9k}).
\]
We are going to use this decomposition to study the properties of the random transfer matrix $T_{10k,E,\omega}$. Take any closed interval $J\subset \R$. Every set of ``frozen'' values $\Xi_1$ defines a random (due to the dependence on $\Xi_2$ and $\Xi_3$) parameter-dependent matrix, that is, a map
\begin{equation}\label{eq:A-map}
A: J \to \SL(2,\R), \quad A(E)= T_{10k,\omega, E}.
\end{equation}

This {transfer} matrix can be decomposed as a product:
\begin{equation}\label{eq:9-10}
T_{10k,\omega, E}=T_{[9k+1,10k],\omega, E}T_{[k+1,9k],\omega, E} T_{[1,k],\omega, E};
\end{equation}
let us consider the interior part of this product, the random transfer matrix $T_{[k+1,9k],\omega, E}$.

We already know due to Lemma~\ref{l:p-applicable} that under the assumptions of Theorem~\ref{t.main}, for every set of values $\Xi_2$ (``frozen'' for the scenario of this lemma), we get a finite set $X=X_{\Xi_2}$ of exceptional energies, defined by Proposition~\ref{p:main} applied to the
conditional law of $(v_{k+1},\dots,v_{9k})$. Hence, the law of the random transfer matrix $T_{[k+1,9k],E}$, conditional to $\Xi_2$, satisfies the measures condition for $E\notin X_{\Xi_2}$.
%transfer matrix $T_{[k+1,9k],E,\omega}$.
 Define now the set $\mE$ to be the ``essential intersection'' of these sets:
\begin{equation}\label{eq:def-mE}
\mE:=\{E \mid E \in X_{\Xi_2} \,\, \text{for $\nu^{2k}$-a.e.} \,\, \Xi_2\}.
\end{equation}
As each individual set $X_{\Xi_2}$ is finite (actually, one can obtain a bound on their cardinality), the same applies to the set~$\mE$. We are now
ready to claim the applicability of the Nonstationary Furstenberg Theorem to the full maps $T_{10k,E,\omega}$. Namely, we have the following lemma.

\begin{figure}
\includegraphics[width=1\textwidth]{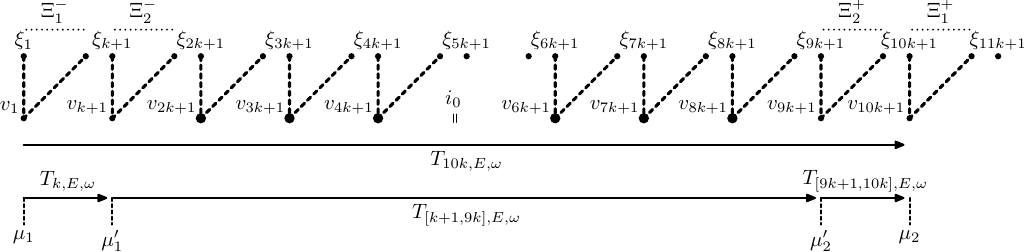}
\caption{Block factors, transfer matrices, and measures $\nu_i, \nu'_i$.}\label{f:block}
\end{figure}

%\newpage

\begin{lemma}\label{l:applicable}
Assume that the assumptions of Theorem~\ref{t.main} hold, and the closed interval $J$ is disjoint from the set $\mE$ defined by~\eqref{eq:def-mE}. Then there exists a compact set $\mK$ of probability measures on set of maps from $J$ to $\SL(2,\R)$, such that for $\nu^{2k}$-a.e. $\Xi_1$ the distribution of the map~$A(\cdot)$, defined by~\eqref{eq:A-map}, belongs to $\mK$, and the conditions~\ref{B:Furstenberg}--\ref{B:Monotonicity} are satisfied for this set.
\end{lemma}

\begin{proof}
Recall that the function $g$ is bounded; let $I_V=[-C_V,C_V]$ be a compact interval so that
$\nu^k$-almost surely we have
\[
g(\xi_{i+1},\dots,\xi_{i+k}) \in I_V.
\]
We will actually construct a compact set $\mK_0$ of probability measures on the set of possible potentials, that is, on the cube $I_V^{10k}$; then, the set $\mK$ is chosen as the image of this set under the push-forward by the continuous map
\begin{equation}\label{eq:subst}
\bar v=(v_1,\dots,v_{10k}) \mapsto A(\cdot), \quad A(E)=T_{10k,\bar v, E}.
\end{equation}

%Decompose the transfer matrix~\eqref{eq:A-map} as a product of three ones:

Note that once both $\Xi_1$ and $\Xi_2$ are fixed, the values of the potential $v_1,\dots,v_{k}, v_{9k+1},\dots, v_{10k}\in I_V^{2k}$ become deterministic, while $\Xi_3$ is independent from these, thus its distribution is equal to $\nu^{7k}$. We now define the set $\mK_1$ of probability measures $\tau$ on
\[
I_V^{2k} \times \R^{2k} \times \R^{7k}
\]
by the following conditions:
\begin{itemize}
\item the law of the second factor is $\nu^{2k}$;
\item conditionally to the second factor being equal to $\nu^{2k}$-a.e. point $\Xi_2$, the first and the third factors are independent, with the law of the third factor that is~$\nu^{7k}$. {We will denote the conditional law for the first factor by $\tau_{\Xi_2}$.}
\end{itemize}

Now, define the set $\mK_0$ as the image of $\mK_1$ under the push-forward by the map $\pi:I_V^{2k} \times \R^{2k} \times \R^{7k} \to I_V^{10k}$,
%sending its point to
defined by
\begin{multline}
\pi: (v_1,\dots,v_k,v_{9k+1},\dots,v_{10k}; \Xi_2 ; \Xi_3) \mapsto
\\ (v_1,\dots,v_k; g(\xi_{k+1},\dots,\xi_{2k}),\dots, g(\xi_{9k},\dots,\xi_{10k-1}) ; v_{9k+1},\dots,v_{10k}).
\end{multline}
%
%\[
%(v_1,\dots, v_{k}, g(\xi_{k+1},\dots,\xi_{2k}), \dots, g(\xi_{8k+1},\dots,\xi_{9k}), v_{9k+1},\dots, v_{10k}).
%\]
The key statement, which justifies the choice of $\mK_1$, $\mK_0$ and $\mK$ above, is the compactness of the image.
\begin{lemma}\label{l:compact}
The image $\mK_0$ defined above is a compact set of probability measures on~$I_V^{10k}$.
\end{lemma}
Let us postpone for now the proof of Lemma~\ref{l:compact}; assuming that it is already established, we complete the proof of Lemma~\ref{l:applicable}.

Namely, fix any compact interval $J\subset \R$ that is disjoint from~$\mE$. The set $\mK$ of possible laws of random parameter-dependent transfer matrices, obtained from the set of distributions $\mK_0$ by push-forward by the map~\eqref{eq:subst}, is then also compact (as a continuous image of a compact set). Every measure from $\mK_0$ is supported on a bounded set $I_V^{10k}$, which implies the assumptions~\ref{B:C1} and~\ref{B:Monotonicity} due to the nature of Schr\"odinger cocycles. Also, notice that all the distributions of the form $(v_1,\dots,v_k, \xi_{k+1}, \xi_{k+2}, \ldots, \xi_{10k}, v_{9k+1},\dots,v_{10k})$ that we can obtain in our setting do belong to $\mK_1$. Indeed, for $\nu^{2k}$-a.e. value of $\Xi_2=(\xi_{k+1}, \ldots, \xi_{2k}, \xi_{9k+1}, \ldots, \xi_{10k})$  and any given $\Xi_1$, the conditional distribution of $(v_1,\dots,v_k,v_{9k+1}, \dots, v_{10k})$  is deterministic and, in particular, is independent of $(\xi_{2k+1},\ldots, \xi_{9k})$.

% Also, as it follows directly from its definition, the set $\mK_0$ contains measures associated to $\nu^{2k}$-a.e. $\Xi_1\in I_V^{2k}$.

%Finally, let us establish condition~\ref{B:Furstenberg}.

Finally (and most importantly), let us check the measure condition~\ref{B:Furstenberg}. Assume that for some energy $E$, the map $f_{T_{10k,E,\omega}}$ almost surely sends some measure~$\msp_1$ on~$\RP^1$ to~$\msp_2$. Using the decomposition~\eqref{eq:9-10}, we can rewrite this as
\begin{equation}\label{eq:center}
(f_{T_{[k+1,9k],\omega, E}})_* \msp'_1 = \msp'_2,
\end{equation}
where
\[
\msp'_1 = (f_{T_{k,\omega, E}})_* \msp_1, \quad \msp'_2 = (f_{T_{[9k+1,10k],\omega, E}^{-1}})_* \msp_2.
\]
For equality~\eqref{eq:center} to hold almost surely, it should hold almost surely conditionally to almost every~$\Xi_2$. Meanwhile, the pair of (random) measures $(\msp'_1,\msp'_2)$ depends only on $v_1,\dots,v_k,v_{9k+1},\dots,v_{10k}$, while conditionally to $\nu^{2k}$-a.e.~$\Xi_2$, the map $(f_{T_{[k+1,9k],\omega, E}})_*$ depends only on $\xi_{2k+1},\dots,\xi_{9k}$. Due to their conditional independence, if~\eqref{eq:center} holds almost surely conditionally to some $\Xi_2$, then, taking any pair $(\bar{\msp}'_1,\bar{\msp}'_2)$ in the support of the random (measure-valued) variable $(\msp'_1, \msp'_2)$, we see that for $\nu^{7k}$-a.e. $\Xi_3$,
\[
(f_{T_{[k+1,9k],\omega, E}})_* \bar{\msp}'_1= \bar{\msp}'_2.
\]
By definition this is possible only if $E\in X_{\Xi_2}$.

Hence, for~\eqref{eq:center} to hold almost surely, one should have $E\in X_{\Xi_2}$ for almost all~$\Xi_2$, which by definition implies $E\in \mE$. As the interval $J$ of parameters is disjoint from~$\mE$, the assumptions~\ref{B:Furstenberg}--\ref{B:Monotonicity} are satisfied. This completes the proof of Lemma~\ref{l:applicable} assuming Lemma~\ref{l:compact}.
\end{proof}

\tikzstyle{box} = [rectangle, minimum width=3cm, minimum height=1cm, text centered, text width=4cm, draw=black]
\tikzstyle{arrow} = [thick,->,>=stealth]

\begin{figure}
\begin{tikzpicture}
%\draw (0,0)--(4cm,4cm);
\node (Kxi) [box]{$\mK_1$: a set of measures on $I_V^{2k}\times \R^{2k}\times \R^{7k}$};
\node (Kv) [box, right of=Kxi, xshift=6cm]{$\mK_0\subset P(I_V^{10k})$: a set of measures, defining random potentials $v_1,\dots,v_{10k}$};
\node (Kprime) [box, below of=Kxi, xshift=0cm, yshift=-2cm]{$\mK_1'\subset \mP=\prod_n P(I_V^{2k})$: a set of sequences of averages over different $U_n$'s};
\node (K) [box, below of=Kv, yshift=-2cm]{$\mK\subset P(C^1(J,\SL(2,\R)))$: a set of measures, defining random parameter-dependent (transfer) matrices};

\draw [arrow] (Kxi) --node[anchor=south] {$\pi_*$} (Kv);
\draw [arrow] (Kv) --node[anchor=west] {$E\mapsto T_{10k,E}$} (K);
\draw [arrow] (Kxi) --node[anchor=east] {bijection $\Phi$} (Kprime);
%\draw [arrow] (Kprime) -- (Kxi);
\draw [arrow] (Kprime) --node[anchor=east, yshift=3mm] {continuous} (Kv);
\end{tikzpicture}
\caption{Passages between different sets of probability measures}\label{f:scheme}
\end{figure}
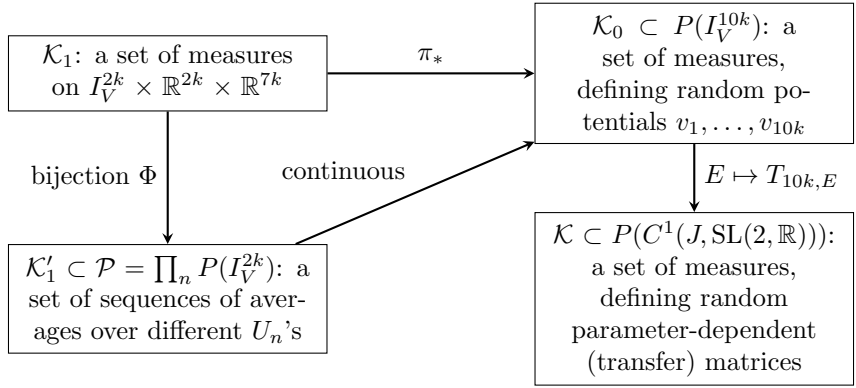

\begin{proof}[Proof of Lemma~\ref{l:compact}]
First, let us reformulate the condition that defines the set $\mK_1$. Namely, given a measure $\tau$ on $I_V^{2k} \times \R^{2k} \times \R^{7k}
$ and a subset $U\subset \R^{2k}$ of positive $\nu^{2k}$-measure, one can consider the conditional measure of $\tau$ on $I_V^{2k} \times U \times \R^{7k}$ and its projections $\tau^{1,3}_U$, $\tau_U$ on $I_V^{2k} \times \R^{7k}$ and on $I_V^{2k}$ respectively. Note that if $\tau\in \mK_1$, then due to its definition,
\begin{equation}\label{eq:mu-U}
\tau^{1,3}_U = \tau_U \times \nu^{7k}.
\end{equation}

{Now, fix a generating sequence of sets $U_n$ that come from a sequence of partitions of $\R^{2k}$, such that each partition is a refinement of the previous one (see Fig.~\ref{f:partitions}). Then, for any refinement
\[
U_n = \bigsqcup_j U_{n_j},
\]
one has
\begin{equation}\label{eq:mu-martingale}
\tau_{U_n} = \sum_j p_j \tau_{U_{n_j}}, \quad \text{where } \, p_j =\frac{\nu^{2k}(U_{n_j})}{\nu^{2k}(U_{n})}.
\end{equation}}
\begin{figure}
\includegraphics[width=0.4\textwidth]{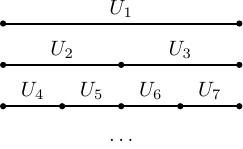}
\caption{A sequence of consecutively refined partitions of~$\R^{2k}$.}\label{f:partitions}
\end{figure}

{For any metric compact $Y$, denote by $P(Y)$ the set of probability measures on~$Y$,
equipped with the transport distance:
\[
\dist_{P(Y)} (\nu_1,\nu_2) = \sup_{\varphi \in \Lip(Y), \atop \|\varphi\|_{\Lip}\le 1} \left| \int_Y \varphi \, d\nu_1 - \int_Y \varphi \, d\nu_2 \right|.
\]
Now, consider the infinite product
\begin{equation}\label{eq:mP-def}
\mP:= \prod_n P(I_V^{2k}),
\end{equation}
equipped with the product topology; then it is a compact, as an (infinite) product of compacts.
Define $\mK_1'\subset \mP$ as the set of sequences of measures that satisfy the relations~\eqref{eq:mu-martingale} for any subpartition of any~$U_n$.
}

{
\begin{lemma}
The set $\mK_1'$ is a compact subset of $\mP$; the map
\[
\Phi:\mK_1\to\mK_1', \quad \Phi:\tau\mapsto (\tau_{U_n})_n,
\]
is a bijection between $\mK_1$ and $\mK_1'$.
\end{lemma}
}
\begin{proof}
{Each of the conditions~\eqref{eq:mu-martingale} defines a closed subset of $\mP$ (it is a linear relation for a finite number of measures). Hence, the set $\mK_1'$ they define is a countable intersection of closed subsets of $\mP$, and thus is a compact.}

{We have already seen that the conditions~\eqref{eq:mu-martingale} indeed hold for the $\Phi$-image of any measure $\tau\in\mK_1$; let us now show that every sequence $(m_n)\in \mK_1'$ admits a unique $\Phi$-preimage $\tau\in\mK_1$, thus showing that $\Phi$ is a bijection. This can be seen by a standard application of the Martingale Convergence Theorem.}

{Namely, take $\Xi_2\in \R^{2k}$ randomly w.r.t. $\nu^{2k}$, and consider the sequence of sets $U_{n_j}\ni \Xi_2$. Then, the sequence of measures $m_{n_j}$, due to the relations~\eqref{eq:mu-martingale}, is a martingale, taking values in the space of probability measures. Due to the Martingale Convergence Theorem, this sequence converges for $\nu^{2k}$-a.e. $\Xi_2\in \R^{2k}$ to some~$\tau_{\Xi_2}$, and the average of $\tau_{\Xi_2}$ over any~$U_n$ is equal to~$m_n$. Now, reconstruct the measure~$\tau$ by using $\tau_{\Xi_2}$ as conditional densities,
\begin{equation}\label{eq:tau-inverse}
\tau:= \int_{\R^{2k}} [\tau_{\Xi_2} \times \delta_{\Xi_2} \times \nu^{7k}] \, d\nu^{2k}(\Xi_2),
\end{equation}
where $\delta_{\Xi_2}$ is the Dirac measure concentrated at the point~$\Xi_2$. One then gets the relation $\tau_{U_n}=m_n$ directly from the averaging property, and~\eqref{eq:mu-U} holds by construction. Hence, $\Phi(\tau)=(m_n)$; finally, it is easy to see that this is the unique $\Phi$-preimage (again using martingale arguments), and thus $\Phi:\mK_1\to \mK_1'$ is a bijection.}
\end{proof}

Finally, note that the map $\pi_*\circ \Phi^{-1}: \mK_1'\to P(I_V^{10k})$ is continuous. To prove it, for $\nu^{2k}$-a.e. $\Xi_2$ consider the probability measure $\vartheta_{\Xi_2}$ on $I_V^{8k}$ that is the conditional (to $\Xi_2$) law of the potentials
\[
(v_{k+1},\dots,v_{9k}).
\]
A standard argument for measurable functions is that the map $\Xi_2\mapsto \vartheta_{\Xi_2}$ is a measurable map to a compact space $P(I_V^{8k})$, and hence for arbitrarily small $\eps>0$ it can be approximated by a piecewise-constant map $\vartheta^{(\eps)}$ in a way that:
\begin{itemize}
\item $\vartheta(\Xi_2)$ and $\vartheta^{(\eps)}(\Xi_2)$ are $\eps$-close to each other except for a set of points $\Xi_2$ of $\nu^{2k}$-measure less than~$\eps$;
\item there is a finite collection of sets $\{U_{n_j}\}_{j=1, \ldots, s}\subset \{U_n\}$ such that $\vartheta^{(\eps)}$ is piecewise-constant with respect to $\{U_{n_j}\}$: there exist measures $m_1,\dots,m_s\in P(I_V^{8k})$ such that
\[
\R^{2k} = \bigsqcup_{j=1}^s U_{n_j}, \quad \text{and}\ \  \forall j=1,\dots,s \ \   \text{we have}\ \
\vartheta^{(\eps)}|_{U_{n_j}} \equiv m_j.
\]
\end{itemize}
Indeed, it suffices to split the compact space $P(I_V^{8k})$ into a finite number $N$ of parts of diameter less then $\eps$; then, the measurable preimage of each of these can be $\frac{\eps}{N}$-approximated by an element of the algebra generated by~$U_n$.

{Let $\tau\in\mK_1$; we will now estimate the distance between its image $\pi_*(\tau)\in \mK_0$ and its approximation }
%Now note that the image of a measure $\tau$ in $P(I_V^{10k})$ is then close to
\begin{equation}\label{eq:approx}
{\pi^{(\eps)}_*(\tau):= \int_{\R^{2k}} [\tau_{\Xi_2} \times \vartheta^{(\eps)}(\Xi_2)] \, d\nu^{2k}(\Xi_2) =}
\sum_j \nu^{2k}(U_{n_j}) \cdot \left(\tau_{U_{n_j}}\times \vartheta^{(\eps)}|_{U_{n_j}}\right).
\end{equation}
{
Indeed, the image $\pi_*$ is equal to the integral
\begin{equation}\label{eq:image}
\pi_*\tau = \int_{\R^{2k}} [\tau_{\Xi_2} \times \vartheta(\Xi_2)] \, d\nu^{2k}(\Xi_2).
\end{equation}
Now, replacing $\vartheta(\Xi_2)$ in~\eqref{eq:image} by  $\vartheta^{(\eps)}(\Xi_2)$ at the points $\Xi_2$ where the distance between them is at most $\eps$ changes the result also at most by $\eps$, and at the points where this distance exceeds $\eps$ changes the result at most by
\[
\diam(I_V^{8k})\cdot \nu^{2k} (\{\Xi_2 \mid \dist(\vartheta(\Xi_2),\vartheta^{(\eps)}(\Xi_2)) \ge \eps) \le (20kC_V) \cdot \eps.
\]
%\begin{multline*}
%\dist \left( \int_{\R^{2k}} [\tau_{\Xi_2} \times \vartheta(\Xi_2)] \, d\nu^{2k}(\Xi_2),
%\end{multline*}
}
{In particular, one has
\begin{equation}\label{eq:pi-dist-eps}
\dist(\pi_*(\tau),\pi^{(\eps)}_*(\tau)) \le (20kC_V+1) \eps.
\end{equation}
}
{On the other hand, for two measures $\tau,\tau'$ we have
\begin{equation}\label{eq:pi-eps-image}
\dist(\pi^{(\eps)}_*\tau,\pi^{(\eps)}_*\tau') \le \sum_{j=1}^s \nu^{2k}(U_{n_j}) \cdot \dist(\tau_{U_{n_j}},\tau'_{U_{n_j}}).
\end{equation}
Joining~\eqref{eq:pi-dist-eps} and~\eqref{eq:pi-eps-image}, we get
\begin{equation}\label{eq:pi-dist}
\dist(\pi_*(\tau),\pi_*(\tau'))\le 2\cdot (20kC_V+1) \eps + \sum_{j=1}^s \nu^{2k}(U_{n_j}) \cdot \dist(\tau_{U_{n_j}},\tau'_{U_{n_j}}).
\end{equation}
}

{The estimate \eqref{eq:pi-dist} implies the continuity of the map $\pi_*\circ \Phi^{-1}:\mK_1'\to P(I_V^{10k})$. Indeed, given any $\eps'>0$, take $\eps>0$ sufficiently small so that $2\cdot (20kC_V+1) \eps < \frac{\eps'}{2}$. Fix the corresponding map $\vartheta^{(\eps)}$ and the elements of our sequence $U_{n_1},\dots,U_{n_s}$ that form a partition of $\R^{2k}$ and such that $\vartheta^{(\eps)}$ is constant on each of these sets. Then, the finite set of inequalities
\begin{equation}\label{eq:dist-tau-U}
\dist(\tau_{U_{n_j}},\tau'_{U_{n_j}})<\frac{\eps'}{2}, \quad j=1,\dots,s
\end{equation}
implies $\dist(\pi_*\tau,\pi_*\tau')<\eps'$. Finally, the distances~\eqref{eq:dist-tau-U} can be estimated in terms of distance in $\mP$ between $\Phi(\tau)$, $\Phi(\tau')$, thus implying the (uniform) continuity of the map $\pi_*\circ \Phi^{-1}$.
%(that are satisfied once the distance between $\tau$ and $\tau'$ in $\mP$ is sufficiently small)
}
%\todo[inline]{Explain that $\eps$ from the points $\Xi_2$ where $\eps$-close plus another $C\eps$ over a small $\nu^{2k}$-measure set of $\Xi_2$'s where $\vartheta$ and $\vartheta^{(\eps)}$ are more than $\eps$-different?}
%In particular, due the approximation~\eqref{eq:approx} allows to compare the images of measures $\tau$ once the measures $\tau_{U_{n_j}}$ used  there are close to each other. In particular, the weak convergence of $\tau_{U_n}$ for every $n$ implies the weak convergence of the images.
\end{proof}

%\newpage

\section{The Proof of the Main Proposition}\label{s.proofofprop}

The goal of this section is to prove Proposition~\ref{p:main}. We start by rewriting the measure condition. Namely, we have the following lemma.

\begin{lemma}\label{l.measureconditionfailure}
A violation of the measure condition for some probability distribution on $\SL(2, \R)$ implies that at least one of the following two statements holds:
\begin{itemize}
\item[(F)] There exist subsets $F, F'\subset \RP^1$, consisting of one or two points, such that $AF=F'$ for almost all matrices~$A$.%\marginpar{I know what is meant, but in the setting of (B1), one typically does not speak of transfer matrices. It may be better to match the formulation of (B1) more closely.}
\item[(SO)] There exist $B_1, B_2\in \SL(2,\R)$ such that ${B_2 A B_1^{-1}}\in \SO(2)$ for almost all  matrices~$A$.
\end{itemize}
\end{lemma}
\begin{proof}
Indeed, assume that one has $(f_A)_*\msp=\msp'$ for some measures $\msp,\msp'$ on $\RP^1$ for almost all  matrices~$A$. If $\msp$ has an atom of weight at least $\frac{1}{2}$, then the sets $F,F'$ of atoms of maximal weight of $\msp$ and $\msp'$ respectively consist of at most two points and satisfy the first conclusion of the lemma. Otherwise, consider the projective line as the boundary of the hyperbolic plane $\bbH$. There are points $p,p'\in \bbH$ that are barycenters of the measures $\msp,\msp'$ respectively. Sending these points to the center of the disc by some maps $B_1,B_2$, we conjugate the action to one preserving the center, which is exactly the action of $\SO(2)$.
 \end{proof}

 %OLD VERSION:
%
%\begin{lemma}
%A violation of the measure condition for some $E$ and some law on the potential $\vec{v}$ implies that at least one of the following two statements holds:
%\begin{itemize}
%\item[(F)] There exist subsets $F, F'\subset \RP^1$, consisting of one or two points, such that $TF=F'$ for almost all transfer matrices~$T$.\marginpar{I know what is meant, but in the setting of (B1), one typically does not speak of transfer matrices. It may be better to match the formulation of (B1) more closely.}
%\item[(SO)] There exist $B_1, B_2\in \SL(2,\R)$ such that $B_1 T B_2^{-1}\in SO(2)$ for almost all transfer matrices~$T$.
%\end{itemize}
%\end{lemma}
%\begin{proof}
%Indeed, assume that one has $T\msp=\msp'$ for some measures $\msp,\msp'$ on $\RP^1$. If $\msp$ has an atom of weight at least $\frac{1}{2}$, then the sets $F,F'$ of atoms of maximal weight of $\msp$ and $\msp'$ respectively consist of at most two points and satisfy the first conclusion of the lemma. Otherwise, consider the projective line as the boundary of the hyperbolic plane $\bbH$. There are points $p,p'\in \bbH$ that are barycenters of the measures $\msp,\msp'$ respectively. Sending these points to the center of the disc by some maps $B_1,B_2$, we conjugate the action to the one preserving the center, that is exactly the action of $\mathrm{SO}(2)$.
% \end{proof}

Notice that it suffices to establish the following key lemma.

\begin{lemma}\label{l:three}
Consider a set $V\subset \R^{d}$ that contains three points, {which satisfy the following property: for some intermediate index $i_0$, we have that for any two of these points, their coordinates differ at least at one position $i<i_0$ and at least at one position $i>i_0+1$}. Then there exists a finite set $X=X_{\xi}$ of energies such that for any $E\notin X$, there are no $p,p'\in \CP^1$ such that for all $\vec{v}\in V$, the transfer matrix $T_{\vec{v},E}$ sends~$p$ to~$p'$.
\end{lemma}

\begin{proof}[Proof of Proposition~\ref{p:main} assuming Lemma~\ref{l:three}]
The conclusion of Lemma~\ref{l:three} handles immediately the part $|F|=1$ of the (F) case. {It also handles the (SO) case, as in that case the image of the point $p=B_1^{-1}([1:i])\in \CP^1$ should almost surely be $p'=B_2^{-1}([1:i])\in \CP^1$.}

Finally, assume that there are sets $F, F'$ of two elements each, and five maps $T_1,\dots,T_5$ that send $F$ to $F'$; fix a point $p\in F$. Then (by the generalized pigeon-hole principle) there exists a point $p'\in F'$ such that $T_i(p)=p'$ for at least three indices $i\in \{1,\dots,5\}$.
\end{proof}

%\newpage

To conclude the proof of Proposition~\ref{p:main} it thus suffices to prove Lemma~\ref{l:three}. We will need a lemma, describing the compositions of the projective maps associated to the Schr\"odinger transfer matrices. For every $a,E$ let
\[
R_{a,E}(z) = \frac{1}{(E-a)-z}
\]
be the projectivization $f_A$ of the  Schr\"odinger matrix
\[
	A= \begin{pmatrix}
                                                 E-a & -1 \\
                                                 1 & 0 \\
	\end{pmatrix},
\]
written in the corresponding affine chart $z=\frac{\psi_2}{\psi_1}$. Also, denote by
\[
	R_{v,E,j} := R_{v_j,E}\circ\dots\circ R_{v_1,E}
\]
the projectivization of the composition of first $j$ maps, associated to the vector of potential $v$. We then have the following lemma.

\begin{lemma}\label{l:eps}
Let $v\in\R^d$ be given. For every $\eps>0$, there exists $r_0>0$ such that for all $|E|>r_0$, the following holds. If for some $z,E$ and some intermediate index $j_1$ one has
\begin{equation}\label{eq:not-in}
R_{v,E,j_1-1}(z)\notin U_{\eps}(E-v_{j_1}),
\end{equation}
then for every $j_2\ge j_1+2$, one has
\begin{equation}\label{eq:in-eps}
\frac{1}{R_{v,E,j_2}(z)} \in U_{\eps}(E-v_{j_2}).
\end{equation}
\end{lemma}

Here $U_{\eps}$ denotes the $\eps$-neighborhood in~$\C$, and in particular, any such neighborhood does not contain the infinity point of $\CP^1$.
\begin{proof}
Without loss of generality, we can assume $\eps\le 1$ (otherwise take $\min (\eps,1)$ instead). Choose
\[
r_0:=1+\max(|v_i|)+ \frac{2}{\eps},
\]
and denote by $(z_j)$ the sequence of images of $z$:
\[
z_j:=R_{v,E,j}(z), \quad j=1,2,\dots.
\]

%\todo[inline]{CLEAN UP INDICES $\pm 1$!!!}
%\end{proof}
%\begin{proof}
By~\eqref{eq:not-in}, we have
\[
|z_{j_1-1}-(E-v_{j_1})|\ge \eps,
\]
and hence
\[
|z_{j_1}| = \frac{1}{|(E-v_{j_1})-z_{j_1-1}|} \le \frac{1}{\eps};
\]
see Fig.~\ref{f:large-E} for an illustration. Next, note that due to the choice of $r_0$, if for some $j$ we have $|z_j|\le \frac{1}{\eps}$, then
\[
|(E-v_{j+1})-z_{j}| \ge r_0 - |v_{j+1}| - \frac{1}{\eps} > \frac{1}{\eps},
\]
and hence
\[
|z_{j+1}|  = \frac{1}{|(E-v_{j+1})-z_{j}|} < \eps.
\]

\begin{figure}
\includegraphics[width=0.9\textwidth]{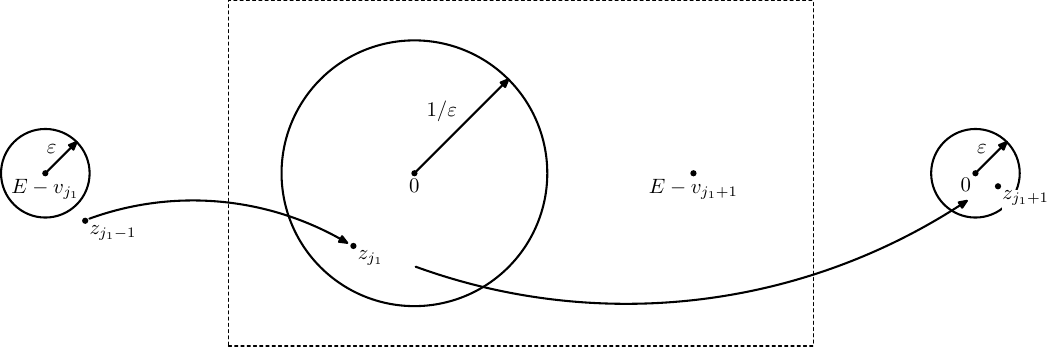}
\caption{Iterating a Schr\"odinger cocycle for large $E$}\label{f:large-E}
\end{figure}

In particular, by induction we have
\[
|z_j|<\eps
\]
for all $j\ge j_1+1$;  applying this for $j=j_2-1$, we get
\[
\frac{1}{R_{v,E,j_2}(z)} = E-v_{j_2}-z_{j_2-1} \in U_{\eps}(E-v_{j_2}),
\]
thus concluding the proof of~\eqref{eq:in-eps}.
\end{proof}

We are now ready to conclude the proof of Lemma~\ref{l:three} and hence of Proposition~\ref{p:main}.

\begin{proof}[Proof of Lemma~\ref{l:three}]

Without loss of generality, we can assume that $V$ consists of exactly three points that satisfy the assumptions of Lemma~\ref{l:three}.

For every pair of different $v,v'\in V$, consider the first and the last coordinates at which they differ:
\[
i_-(v,v'):=\min (i \mid v_i \neq v'_i), \quad i_+(v,v'):=\max (i \mid v_i \neq v'_i).
\]
Let % $\eps(v, v')>0$
\[
\eps(v,v') := \min (|(v-v')_{i_-(v,v')}|, |(v-v')_{i_+(v,v')}|)
\]
be a lower bound for the corresponding differences. Finally, set
\[
\eps:=\frac{1}{3} \min \{ \eps(v,v') \mid v,v'\in V, \,\, v\neq v' \}.
\]
Applying Lemma~\ref{l:eps} for this $\eps$ and each of the three points $v\in V$, we get values $r_0=r_0(v)$. Let us show that for $R:=\max_{v\in V} r_0(v)$, the conclusion of Lemma~\ref{l:three} holds for all $|E|>R$.

Indeed, take any two $v,v'\in V$ and consider the index $j_1:=i_-(v,v')$. Let us %To do so, we will
obtain a contradiction to the assumption that for
some $p, p'\in \CP^1$
\[
\forall v\in V \quad T_{v,E}\ p = p'.
\]
As previously, let $z$ be the (possibly infinite) coordinate of $p$ in the chart $\frac{\psi_2}{\psi_1}$. Take any pair of different $v,v'\in V$, and let $j_1$ be the corresponding index of first difference $j_1:=i_-(v,v')<i_0$. Note that as $|v_{j_1}-v'_{j_1}|>2\eps$, one has
\[
U_{\eps}(E-v_{j_1}) \cap U_{\eps}(E-v'_{j_1}) = \emptyset,
\]
hence the common image
\[
z_{j_1-1} = R_{v,E,j_1-1}(z) = R_{v',E,j_1-1}(z)
\]
does not belong to at least one of these neighborhoods. Hence, for any pair $v,v'\in V$, the assumption of Lemma~\ref{l:eps} holds for some $j_1<i_0$ for at least one of them, and hence holds for all $v\in V$, except at most one. As $V$ contains three points, there exist two $v,v'\in V$, for both of which the estimate~\eqref{eq:in-eps} holds for all $j_2>i_0+1$.

Now, take $j_2=i_+(v,v')>i_0+1$. Then, from the assumption
\[
T_{v,E}\ p = T_{v',E}\ p = p'
\]
we see that the images of $p$ at the index $j_2+1$ coincide, as both are preimages of $p'$ under coinciding parts of the potential; hence,
\[
\frac{1}{R_{v,E,j_2}(z)} = \frac{1}{R_{v',E,j_2}(z)}.
\]
At the same time, from the conclusion of Lemma~\ref{l:eps} for $v$ and $v'$, this value should belong to each of the neighborhoods
\[
U_{\eps}(E-v_{j_2}) \quad \text{and} \quad U_{\eps}(E-v'_{j_2}),
\]
and these neighborhoods are disjoint due to the choice of~$\eps$. We have obtained the desired contradiction.
%At this moment, the image of
%Note that for every $v,v'\in V, \,\, v\neq v'$ and any initial point $p\in \CP^1$ the assumptions of Lemma~\ref{l:eps} hold at least for one of iterations also
\end{proof}

\section{Proof of Localization}

Let us bring together all the pieces of the puzzle and complete the proof of Theorem \ref{t.main} and Theorem \ref{t.dynloc}.

\subsection{Proof of Spectral Localization}

Here we show that  Theorem~\ref{t.vector} combined with Lemma \ref{l:applicable} implies spectral localization. The proof mimics the arguments of the proof of Theorem 1.1 from \cite{GK3}.

We will need the following result, which is usually referred to as ``Shnol's Theorem'', due to a similar result in the paper \cite{Sch} (see also \cite{Gl1, Gl2}):

\begin{theorem}[Shnol's Theorem]\label{t.Schnol}
Let $H: \ell^2(\Z)\to \ell^2(\Z)$
be an operator of the form
\[
[H \psi](n) = \psi(n-1) + \psi(n + 1) + V (n)\psi(n),
\]
with a bounded potential $\{V (n)\}_{n\in \Z}$, and let $B\subseteq \R$ be any Borel subset of the real line.  If every polynomially bounded solution to $H \psi = E\psi$, $E\in B$, is in fact exponentially decreasing, then $H$  has pure point spectrum on $B$, with exponentially decaying eigenfunctions.

A similar statement holds for operators on $\ell^2(\N)$ with a Dirichlet boundary condition.
\end{theorem}

The formal proof in the discrete case can be found, for instance, in \cite[Theorem~7.1]{Kir} and \cite[Theorem~2.4.2]{DF22}; we also refer the reader to some improved versions of this result in~\cite[Lemma~2.6]{JZ} or~\cite{H}. In the continuum case, Theorem~\ref{t.Schnol} follows also from \cite[Theorem~1.1]{Sim}.

\begin{proof}[Proof of Theorem~\ref{t.main}]
Let us consider the finite set $\mE\subset \R$ defined by (\ref{eq:def-mE}). The complement $\R\backslash \mE$ can be represented as a countable union of compact intervals. {Since almost surely, the finite set $\mE$ carries no weight with respect to the spectral measure,} it is enough to show that on each such compact interval, almost surely the operator (\ref{e.model}) has pure point spectrum {with exponentially decaying eigenfunction}.

Let us fix a closed interval $J\subset  \R\backslash \mE$. Split the sequence $\{\xi_n\}$ into blocks of length $10k$,
$$
\{\xi_n\}_{n\in\Z}=\bigcup_{j\in \Z}\{\xi_{10kj+1}, \xi_{10kj+2}, \ldots, \xi_{10k(j+1)}\}.
$$
Fix any sequence $\bar a=\{\bar a_j\}$ of vectors $\bar a_j=(a_{1, j}, a_{2, j}, \ldots, a_{k, j})\in (\text{supp}\,\nu)^k$, $j\in \Z$, and set the first $k$ values of the $j$-th block to be $(a_{1, j}, a_{2, j}, \ldots, a_{k, j})$, therefore ``freezing'' the values of $\{\xi_{10kj+1}, \xi_{10kj+2}, \ldots, \xi_{10kj+k}\}$.

Consider the (conditional) random potential
$ V_{\bar a}(n)=        g(\xi_n, \xi_{n+1}, \ldots, \xi_{n+k-1}),$ and denote by $H_{\bar a}$ the random operator defined by potential $V_{\bar a}(n)$. Denote by $A_{j, \bar a, E}$, $E\in J$, the product of transfer matrices over the $j$-th block of the potential values:
$$
A_{j, \bar a}=\left(
               \begin{array}{cc}
                 E-V_{\bar a}(10k(j+1)) & -1 \\
                 1 & 0 \\
               \end{array}
             \right) \cdots \left(
               \begin{array}{cc}
                 E-V_{\bar a}(10kj+1) & -1 \\
                 1 & 0 \\
               \end{array}
             \right).
$$
Notice that the matrices $\{A_{j, \bar a}\}$ are independent, and, due to Lemma \ref{l:applicable}, satisfy all the assumptions of  Theorem~\ref{t.vector}. This implies that any polynomially bounded solution of the equation $H_{\bar a} \psi = E\psi$, $E\in J$, is in fact exponentially decreasing. Hence, due to Shnol's theorem, the operator $H_{\bar a}$ almost surely (conditionally on the ``frozen'' values) has pure point spectrum in $J$ with exponentially decaying eigenfunctions.

Finally, an application of Fubini's Theorem implies that the operator $H$ almost surely has pure point spectrum in $J$ {with exponentially decaying eigenfunctions.}
\end{proof}

\subsection{Proof of Dynamical Localization}

Here we justify Theorem \ref{t.dynloc}.
\begin{proof}[Proof of Theorem \ref{t.dynloc}]
Let us fix a compact interval  $J \subset \R \setminus \mathcal{E}$. We need to establish that the operator $H$ defined by (\ref{e.model}) satisfies the following property:
%there is $\beta > 0$ such that for any $\varepsilon > 0$ and $\nu^{\Z}$-almost surely we have the following: there is a constant $C > 0$ such that
%\begin{equation}\label{eq:edl}
%	\sup\limits_{t \in \R}
%	|\langle\delta_n, e^{-itH} \chi_{J}(H) \delta_m\rangle|
%	\leq C e^{\epsilon |m|} e^{-\beta|n-m|}
%\end{equation}
%for all $m,n \in \Z$.

\begin{defi}
Let $H$ be a self-adjoint operator on $\ell^2(\mathbb{Z})$. The operator $H$ has {\it semi-uniform dynamical localization (SUDL)} on a closed interval $J\subset \R$ if there is $\beta>0$ such that for any $\eps>0$ there is a constant $C_\eps>0$ so that for all $n, m\in \mathbb{Z}$
$$
	\sup\limits_{t \in \R}|\langle\delta_n, e^{-itH}\chi_{J}(H)\delta_m\rangle|\le C_\eps e^{\eps |m|-\beta|n-m|}.
$$
\end{defi}
We will argue that $H$ has a different property, namely (SULE):

\begin{defi}
A self-adjoint operator $H:\ell^2(\mathbb{Z})\to \ell^2(\mathbb{Z})$ has {\it semi-uniformly localized eigenfunctions (SULE)} on a closed interval $J\subset \R$ if $H$ has a set $\{\phi_n\}_{n=1}^\infty$ of orthonormal eigenfunctions {that is complete in the sense that the closure of its span is $\mathrm{Ran} \, \chi_J(H)$}, and there are $\beta>0$ and $\mmax_n\in \mathbb{Z}$, $n\in \mathbb{N}$, such that for each $\eps>0$ there exists a constant $C_\eps$ {so that}
\begin{equation}\label{eq:SULE}
 |\phi_n(m)|\le C_\eps e^{\eps|\mmax_n|-\beta|m-\mmax_n|}
\end{equation}
 for all $m\in \mathbb{Z}$ and $n\in \mathbb{N}$.
\end{defi}
{Theorem 7.5 from \cite{DJLS} shows that (SULE) $\Rightarrow$ (SUDL). In fact, this result is local in energy. A very minor modification of the proof of \cite[Theorem 7.5]{DJLS} given in that paper shows that if (SULE) holds on some interval $J$, then (SUDL) holds on $J$. The only changes necessary are to let the sum in the first display of that proof run over the (SULE) orthonormal basis of $\mathrm{Ran} \, \chi_J(H)$ and to then note that \cite[(7.4)]{DJLS} still holds.}\footnote{{In order to see the latter fact, one observes that while \cite[(7.2b)]{DJLS} will now in general fail, one still has $\sum_n |\phi_n(m)|^2 \le 1$ by Bessel's inequality, and only this inequality is needed when deriving \cite[(7.4)]{DJLS}.}}

Now, repeating the steps from the proof of Theorem \ref{t.main}, let us fix any sequence $\bar a=\{\bar a_j\}$ of vectors $\bar a_j=(a_{1, j}, a_{2, j}, \ldots, a_{k, j})\in (\text{supp}\,\nu)^k$, $j\in \Z$, and use it to ``freeze'' the values of $\{\xi_{10kj+1}, \xi_{10kj+2}, \ldots, \xi_{10kj+k}\}$. The random matrices $A_{j, \bar a, E}$, $E\in J$, are independent and  satisfy the conditions (B1)--(B3) from Section \ref{ss.introparamfurst}. Therefore, almost surely every eigenvector of $H_{\bar a}$ corresponding to an eigenvalue from $J$ must decay exponentially. Moreover, a verbatim repetition of the arguments from \cite[Section~4.4]{GK3} applied to the product of the random matrices $A_{j, \bar a, E}$ shows that if one takes $\beta=\frac{h}{4}$, where $h>0$ is given by Proposition \ref{t.vector}, then almost surely for any $\eps>0$, there exists $C_\eps>0$ such that for every eigenfunction $\psi$ with an eigenvalue in $J$, there exists $\widehat m\in \N$ such that
$$
|\phi(m)|\le C_\eps e^{\eps|\mmax|-\beta|m-\mmax|}
$$
for all $m\in \Z$.

Notice that the value $h>0$ can be chosen uniformly in the choice of $\bar a$. Therefore, an application of Fubini's Theorem implies that the operator $H$ almost surely has (SULE) (and hence (SUDL)) on the interval $J$.
\end{proof}

\section{Exceptional Energies: Example }\label{s.last}

%\begin{remark}\label{r.zeroLE}
A priori it is not obvious that allowing for the existence of exceptional energies in Proposition~\ref{p:main} is not an artifact of the proof. While the existence of exceptional energies is well known in the random dimer model \cite{DWP} (see also \cite[Section~5.11]{DF24}), the random dimer model does not formally fit the framework of this paper\footnote{It does, however, fit the more general framework of \cite{ADZ1, ADZ2}, as discussed in \cite{ADZ1}.} (but see Remark \ref{r.r}, part c)). In this section we demonstrate that in our setting, exceptional energies also cannot be avoided in general.
%\end{remark}

Let us set $\{\xi_i\}$ to be an i.i.d.\ sequence of zeroes and ones chosen randomly with the same probability~$1/2$. In other words, $\mathrm{supp} \, \nu = \{ 0,1 \}$ and $\nu(\{0\}) = \nu(\{1\}) = \frac12$. Consider $g:\mathbb{R}^2\to \mathbb{R}$ defined by
\begin{equation}\label{eq:g}
g(\xi_n, \xi_{n+1})=\xi_n-\xi_{n+1}.
\end{equation}

\begin{prop}\label{p.example}
The Lyapunov exponent of the Schr\"odinger cocycle defined by the random potential given by~\eqref{e.model} and~\eqref{eq:g} vanishes at the energy value $E=0$. Moreover, almost sure semi-uniform dynamical localization (SUDL) without projecting away from the exceptional energies fails for the associated family of Schr\"odinger operators.
\end{prop}

\begin{remark}
(a) More precisely, in this example we see the necessity of projecting away from the exceptional energies. That is,  the conclusion of Theorem~\ref{t.dynloc} (i.e., the bound \eqref{eq:edl}) cannot hold without the spectral projection $\chi_J(H)$. This follows for the example under consideration both in the whole-line case and in the half-line case.
\\[1mm]
(b) The underlying reason is that at energy $E = 0$, the transfer matrices almost surely are bounded by a stretched exponential. This is akin to one of the standard examples where the classical Furstenberg theorem fails --- random products of
$$
\left(
               \begin{array}{cc}
                 2 & 0 \\
                 0 & 1/2 \\
               \end{array}
             \right) \quad \text{and} \quad \left(
               \begin{array}{cc}
                 1/2 & 0 \\
                 0 & 2 \\
               \end{array}
             \right).
$$
\\[1mm]
(c) The examples in part (b) are a special case of one of the two settings in which the classical Furstenberg theorem for random products of $\mathrm{SL}(2,\R)$ matrices fails -- rotation-valued cocycles and cocycles whose projective action preserves a set of cardinality one or two; compare Lemma~\ref{l.measureconditionfailure}. Observe that in the example in part (b), there is a set of two directions that is preserved. Random Schr\"odinger operators with exceptional energies with the other types of exceptional behavior are observed in the random dimer model \cite{DWP} (rotation-valued) and in the random Kronig-Penney model \cite{DKS} (preserving a single direction). The discreteness of the set of exceptional energies in these general settings is guaranteed by \cite{BDFGVWZ2, DFHKS, DSS}.
\\[1mm]
(d) In the random dimer model and the random Kronig-Penney model, one can use the general method of Damanik-Tcheremchantsev \cite{DT} to prove positivity of suitably chosen transport exponents that measure transport on a power-law scale. This is because this method applies whenever the transfer matrices are bounded from above by a power of the distance, which is clearly the case in the elliptic and parabolic cases which arise in these two settings. In the present example, however, the exceptional matrices are hyperbolic, and hence the main result of \cite{DT} does not apply directly since it assumes power-law estimates at some energy, whereas all we have is a stretched exponential. In the proof below we show that while the main result of \cite{DT} does not apply, the method itself can still be used to establish the failure of \eqref{eq:edl} as soon as the spectral projection $\chi_J(H)$ is removed.
\end{remark}

\begin{proof}
Note  that individual transfer matrices of the Schr\"odinger cocycle can be rewritten as
\[
\Pi_{j, E}=
\left(
               \begin{array}{cc}
                 1 & E-V(j) \\
                 0 & 1 \\
               \end{array}
             \right)
\left(
               \begin{array}{cc}
                 0 & -1 \\
                 1 & 0 \\
               \end{array}
             \right);
\]
for $E=0$ and $V(j)=\xi_{j+1}-\xi_{j}$ this can be rewritten as
\[
\Pi_{j, 0}=
\left(
               \begin{array}{cc}
                 1 & -\xi_{j+1} \\
                 0 & 1 \\
               \end{array}
             \right)
\left(
               \begin{array}{cc}
                 1 & \xi_j \\
                 0 & 1 \\
               \end{array}
             \right)
\left(
               \begin{array}{cc}
                 0 & -1 \\
                 1 & 0 \\
               \end{array}
             \right).
\]
Let
\[
R=\left(
               \begin{array}{cc}
                 0 & -1 \\
                 1 & 0 \\
               \end{array}
             \right), \quad
P_a=\left(
               \begin{array}{cc}
                 1 & a \\
                 0 & 1 \\
               \end{array}
             \right),
\]
and consider the conjugates
$$
R_a:=P_a R P_a^{-1} = \left(
               \begin{array}{cc}
                 a & -a^2-1 \\
                 1 & -a \\
               \end{array}
             \right).
$$

\newpage

Then the transfer matrix $T_{n,\omega,E}$ for $E=0$ can be rewritten as
\begin{align}\label{eq:R-xi}
T_{n,\omega,0} & = (P_{-\xi_{n+1}} P_{\xi_n} R) \cdot (P_{-\xi_{n}} P_{\xi_{n-1}} R) \cdot (P_{-\xi_{n-1}} P_{\xi_{n-2}} R) \dots (P_{-\xi_{2}} P_{\xi_1} R)
\\
\nonumber & = P_{-\xi_{n+1}} \cdot (P_{\xi_n} R P_{\xi_{n}}^{-1}) \cdot (P_{\xi_{n-1}} R P_{\xi_{n-1}}^{-1}) \dots  (P_{\xi_1} R P_{\xi_1}^{-1}) \cdot P_{\xi_1}
\\
\nonumber & = P_{-\xi_{n+1}} \cdot R_{\xi_n} R_{\xi_{n-1}} \dots R_{\xi_1} \cdot P_{\xi_0},
\end{align}
thus up to two bounded factors it is a product of independently chosen $R_{\xi_j}$.

Now, the matrices $R_0$, $R_1$ are conjugate to $R$, so their squares are equal to minus the identity. Hence, {when writing a product of factors of the form} $R_0$ and $R_1$, if the same matrix appears twice in a row, up to a sign it can be cancelled out. The length of the non-cancelled word thus changes by $+1$ and $-1$ equiprobably. Hence, the log-norm of $T_{n,\omega,0}$ can be estimated as %\marginpar{Why not $2C_P + C_R |S_{n,\omega}|$?}
\begin{equation}\label{eq:log-norm-T-S}
\log \|T_{n,\omega,0} \| \le 2C_P + C_R |S_{n,\omega}|,
\end{equation}
where $S_{n,\omega}$ is an equiprobable $+1/-1$ random walk on $\Z$, and
\[
C_P :=  \max_{a=0,1} (\log \|P_a\|), \quad C_R :=  \max_{a=0,1} (\log \|R_a\|).
\]

Applying the law of iterated logarithm to $S_n$ (see, for instance, \cite[Chapter~8, Theorem~1.1]{G}), we get that almost surely for all $n$ sufficiently large one has
\begin{equation}\label{eq:S-log}
|S_{n,\omega}| < c \sqrt{n \log \log n},
\end{equation}
where $c>1$ is a constant (for instance, one can take $c=2$). Dividing~\eqref{eq:log-norm-T-S} by $n$ and using~\eqref{eq:S-log}, we thus get that almost surely
\[
\lim_{n\to\infty} \frac{1}{n} \log \| T_{n,\omega,0}\| =0,
\]
thus concluding the proof of the first part of Proposition \ref{p.example}.

Let us address the second part of the proposition. Recall that $T_{[m,n],\omega,E}$ denotes the transfer matrix from $m$ to $n$ (and random parameter $\omega$ and energy $E = 0$), and observe that%\marginpar{Should the notation $T_{[m,n],\omega,E}$ be used for all matrices here? Also, how exactly are they defined? Do they map from $m$ to $n$ or is this the indicated range of the factors? The former case is more standard, and in that case $T_{[m,n],\omega,0} = T_{n,\omega,0}T_{m-1,\omega,0}^{-1}$ would become $T_{[m,n],\omega,0} = T_{[0,n],\omega,0}T_{[0,m],\omega,0}^{-1}$.}
\[
T_{[m,n],\omega,0} = T_{n,\omega,0}T_{m-1,\omega,0}^{-1}, \quad \|T_{m-1,\omega,0}^{-1}\| = \|T_{m-1,\omega,0}\|.
\]
Taking a sufficiently large constant $C'$, we get from~\eqref{eq:log-norm-T-S}, joined with~\eqref{eq:S-log}, that almost surely for all sufficiently large $N$,
\begin{equation}
\forall n, \quad |n|\le N \quad \log \|T_{n,\omega,0}\| \le C' \sqrt{N \log\log N},
\end{equation}
and hence that
\begin{equation}\label{e.strexpbound}
\forall m,n, \quad |m|,|n|\le N \quad \log \|T_{[m,n],\omega,0}\| \le 2C' \sqrt{N \log\log N}.
\end{equation}

Now, let us extend the above upper bound for the norms to close enough energy levels. Namely, recall the following estimate from \cite{DT}:

\begin{lemma}[Lemma~2.1 in \cite{DT}]\label{first}
Let $E \in {\R}, N>0$. Define
$$
K(N)=\sup_{|n|\le N, |m| \le N} \|T_{[m,n],\omega,E}\|
$$
in the case of $\ell^2(\Z)$ and
$$
L(N)=\sup_{1 \le n,m \le N} \|T_{[m,n],\omega,E}\|
$$
in the case of $\ell^2(\N)$. Let $\delta \in \C$. Then, the following bounds hold:
\begin{align*}
\|T_{n,\omega,E+\delta}\| & \le K(N) \exp (K(N) |n| |\delta|), \  |n|\le N,   %1\le n \le N, \\
%\|T_{n,\omega,E+\delta}\| & \le K(N) \exp (K(N) |n| |\delta|), \ -N \le n \le 0,
\end{align*}
in the case of $\ell^2(\Z)$ and
$$
\|T_{n, \omega,E+\delta}\| \le L(N) \exp (L(N) |n| |\delta|), \ 1 \le n \le N, \ \
$$
in the case of $\ell^2(\N)$.
\end{lemma}

For the $\omega$'s in question we therefore have $\log K(N) \le 2C' \sqrt{N \log \log N}$ and $\log L(N) \le 2C' \sqrt{N \log \log N}$ due to~\eqref{e.strexpbound}. For definiteness, let us discuss the whole line case (the modifications for the half line case are easy).

We will fix a sufficiently large time scale $T$, and will choose a large $N$ to be related to it by
\begin{equation}\label{eq:N-T}
N = N(T) := (\log T)^{1.9}
\end{equation}
(we will explain this choice later; actually, the exponent $1.9$ can be replaced by any number in the open interval $(1,2)$).

Now, we are going to estimate the part of the time-averaged norm of the initial atomic vector $\delta_1$ that was transported in this time scale outside the box $\{|n| < N\}$. Namely, consider the integral
\begin{equation}\label{eq:transported}
\sum_{|n| \ge N} \frac{1}{T} \int_0^\infty e^{- \frac{2t}{T}} \left| \langle \delta_n , e^{-itH_\omega} \delta_1 \rangle \right|^2 \, dt.
\end{equation}

The choice~\eqref{eq:N-T} ensures that  $K(N) N \lesssim T$; hence, taking $\delta$ of order $|\delta| \sim \frac{1}{T}$,%\marginpar{Small change here: it is better to refer to $\delta$ as the change in energy, rather than the energy window width, because it may not even be real. In the derivation in \cite{DT}, $E + \delta$ consists of "energies" of the form $E' + \frac{i}{T}$, $E' \in \R$, $|E'-E| \le \frac{1}{T}$. I have added a footnote below to emphasize this point.}
we get that the upper estimate in Lemma~\ref{first} does not exceed
\begin{equation}\label{e.K(N)bound}
K(N) \exp (K(N) N |\delta|) \lesssim e^{C \sqrt{N \log \log N}}.
\end{equation}
%
%
%Of interest now are $\delta$'s of order $|\delta| \sim \frac{1}{T}$ and $N$'s given by
%$$
%N = N(T) := (\log T)^{1.9}
%$$
%for $T \ge 2$ (the exponent $1.9$ can be replaced by any number in the open interval $(1,2)$).
With these estimates in place, one can now mimic the steps in the proof of \cite[Theorem~1]{DT} to obtain a lower bound for the averaged transported mass~\eqref{eq:transported}:\footnote{In this derivation, $E + \delta$ will be of the form $E' + \frac{i}{T}$, $E' \in \R$, $|E'-E| \le \frac{1}{T}$. The steps of the derivation are the same -- the difference is that the upper bound $N^\alpha$ used in \cite{DT} needs to be replaced by $e^{C \sqrt{N \log \log N}}$ (compare \cite[Equation~(36)]{DT} and \eqref{e.K(N)bound} above), which has the result that the factor $N^{-2\alpha}$ in \cite[Equation~(41)]{DT} has to be replaced by $e^{-2C \sqrt{N \log \log N}}$.}
$$
\sum_{|n| \ge N} \frac{1}{T} \int_0^\infty e^{- \frac{2t}{T}} \left| \langle \delta_n , e^{-itH_\omega} \delta_1 \rangle \right|^2 \, dt \gtrsim \frac{N}{T} e^{- 2 C \sqrt{N \log \log N}}.
$$
Now,~\eqref{eq:N-T} implies that $\sqrt{N \log \log N} <\log T$ for all sufficiently large $T$, and hence
\begin{equation}\label{e.lowerbound}
\sum_{|n| \ge (\log T)^{1.9}} \frac{1}{T} \int_0^\infty e^{- \frac{2t}{T}} \left| \langle \delta_n , e^{-itH_\omega} \delta_1 \rangle \right|^2 \, dt \gtrsim \frac{(\log T)^{1.9}}{T^{1 + 2 C}} .
\end{equation}

The estimate~\eqref{e.lowerbound} implies that the bound \eqref{eq:edl} in the conclusion of Theorem~\ref{t.dynloc} cannot hold without the spectral projection $\chi_J(H)$. Indeed, otherwise one would have almost surely that
%Since
%\begin{equation}\label{e.limit}
%\lim_{T \to \infty} \frac{\frac{(\log T)^{1.9}}{T^{1 + 2 C}}}{e^{- 2 \beta (\log T)^2}} = \infty
%\end{equation}
%for every $\beta > 0$, this shows that in this model, the conclusion of Theorem~\ref{t.dynloc} (i.e., the bound \eqref{eq:edl}) cannot hold without the spectral projection $\chi_J(H)$.
%
%Indeed, assume that \eqref{eq:edl} holds without the spectral projection $\chi_J(H)$ for some $\beta > 0$. Then, almost surely,
$$
\sup\limits_{t \in \R}
	|\langle\delta_n, e^{-itH} \delta_1 \rangle| \lesssim e^{-\beta|n-1|}.
$$
Squaring, summing over $|n|\ge N$ and integrating with the weight $\frac{1}{T} e^{-\frac{2t}{T}}$, we see that this would imply an exponential upper bound for the time-averaged transported mass:
\begin{equation}\label{e.upperbound}
\sum_{|n| \ge (\log T)^{1.9}} \frac{1}{T} \int_0^\infty e^{- \frac{2t}{T}} \left| \langle \delta_n , e^{-itH_\omega} \delta_1 \rangle \right|^2 \, dt \lesssim e^{- 2 \beta (\log T)^{1.9}}.
\end{equation}
And the upper bound~\eqref{e.upperbound} is incompatible with the lower bound~\eqref{e.lowerbound} for $T$ large enough, since
\[
e^{- 2 \beta (\log T)^{1.9}} = o \left(\frac{(\log T)^{1.9}}{T^{1 + 2 C}}\right)
\]
as $T\to\infty$.
\end{proof}

\section*{Acknowledgments}
The authors are grateful to S.\,Jitomirskaya for useful discussions and several suggestions and remarks on the first draft of the text that allowed us to enhance the presentation.  %raising the question leading to Remark 1.3 (g) and for encouraging us to expand Section 6 by providing an argument showing that not only dynamical localization cannot be shown by our methods without excluding exceptional energies, but that SUDL actually does not hold in this case.

\end{document}